\documentclass[preprint,12pt]{elsarticle}
\usepackage{amsthm,amsfonts,amssymb,amscd,amsmath,enumerate,verbatim,calc,graphicx,geometry}
\usepackage[all]{xy}
\newtheorem{theorem}{Theorem}[section]
\newtheorem{lemma}[theorem]{Lemma}
\newtheorem{proposition}[theorem]{Proposition}
\newtheorem{corollary}[theorem]{Corollary}
\theoremstyle{definition}
\theoremstyle{definitions}
\newtheorem{definition}[theorem]{Definition}

\newtheorem{remark}[theorem]{Remark}
\newtheorem{example}[theorem]{Example}
\theoremstyle{notations}

\theoremstyle{remarks}

\newcommand{\N}{\mathbb{N}}

\newcommand{\sub}{\subseteq}

\newcommand{\lo}{\longrightarrow}

\newcommand{\wt}{\widetilde}
\newcommand{\vf}{\varphi}

\newcommand{\al}{\alpha}

\newcommand{\bt}{\beta}
\newcommand{\ti}{\tilde}

\newcommand{\psg}{\pi_1^{sg}(X,x)}
\newcommand{\psp}{\pi_1^{sp}(X,x)}

\newcommand{\pc}{p:\wt{X}\lo X}
\newcommand{\pst}{p_*\pi_1(\wt{X},\ti{x})}
\newcommand{\V}{\mathcal{V}}
\newcommand{\U}{\mathcal{U}}

\journal{???}

\begin{document}

\begin{frontmatter}

\title{On the Existence of Categorical Universal Coverings}

\author[]{Ali~Pakdaman\corref{cor1}}
\ead{a.pakdaman@gu.ac.ir}
\address{Department of Mathematics, Faculty of Sciences, Golestan University,
P.O.Box 155, Gorgan, Iran.}
\author[]{Hamid~Torabi}
\ead{hamid$_{-}$torabi86@yahoo.com}
\author[]{Behrooz~Mashayekhy}
\ead{bmashf@um.ac.ir}
\address{Department of Pure Mathematics,\\ Center of Excellence in Analysis on Algebraic Structures,\\ Ferdowsi University of Mashhad,\\
P.O.Box 1159-91775, Mashhad, Iran.}
\cortext[cor1]{Corresponding author}

\begin{abstract}
In this paper, we study necessary and sufficient conditions for the existence of categorical universal coverings using open covers of a given space $X$.
 As some applications, first we present a generalized version of the Shelah Theorem (Mycielski's conjecture: If $X$ is a Peano continuum, then $\pi_1(X,x)$ is
uncountable or $X$ has a simply connected universal covering) which states that a first countable Peano space has a categorical universal covering or has an uncountable fundamental group.
Second, we prove that the one point union $X_1\vee X_2=\frac{{X_1}\cup {X_2}}{{x_1}\sim {x_2}}$ has a categorical universal covering if and only if both $X_1$ and $X_2$ have categorical universal coverings.
\end{abstract}


\end{frontmatter}

%\\\\\\\\\\\\\\\\\\\\\\\\\\\\\\\\\\\\\\\\\\\\\\\\\\\\\\\\\\\\\\\\\\\\\\\\\\\\\\\\\\\\\\\\\\\\\\\\\\\\\\\\\\\\\\\\\\\\\\\\\\\\\\\\\\\\\\\\\
%=========================================================================================================================================
%/////////////////////////////////////////////////////////////////////////////////////////////////////////////////////////////////////////
\section{Introduction and motivation}
Unlike modern nomenclature, the term universal covering
space will always means a categorical covering space that is a covering
$p:\wt{X}\lo X$ with the property that for every covering $q:\wt{Y}\lo X$ with a path connected space $\wt{Y}$ there exists a unique (up to equivalence) covering $f:\wt{X}\lo\wt{Y}$ such that $q\circ f= p$.

Nowadays, lots of literatures can be found about the covering spaces and their relations with fundamental groups and almost all of them proceed on the classifying of covering spaces, using the universal covering spaces. Also, the important role of the universal covering spaces in the geometry causes that finding features on a given space in which guaranties the existence of the universal covering space can be a challenge.

Simply connected covering spaces are examples of universal coverings that have been studied more and partially sufficient. As can be seen in many textbooks, locally path connectedness and semi-locally simply connectedness of a given space $X$ is equivalent to the existence of simply connected universal covering spaces \cite{Lu,Sp}. But for the spaces that do not have these nice local behaviors, the existence of simply connected universal coverings is not possible. The question that naturally arises here is that: Can we provide conditions that ensure the existence of (categorical) universal coverings for spaces with bad local behaviors?

In this regard, the deep connection between fundamental groups and covering spaces becomes more evident. Recently, with the emergence of new subgroups of the fundamental group that will be born in the absence of semi-locally simply connectedness, studying the existence of universal coverings is more accessible. For example, the authors \cite{P2,T2,P5} have introduced universal covering spaces of some these locally complicated spaces.

Our first main result of this paper, Theorem \ref{T1}, introduces equivalent conditions, from various viewpoints, for the existence of universal coverings. The main idea is working with the Spanier groups with respect to open covers of a given space $X$ which have been introduced in \cite{Sp} and named in \cite{FRV}. The importance of these groups and their intersection which is named Spanier group, $\psp$, is studied by H. Fischer et al. in \cite{FRV} in order to modification of the definition of semi-locally simply connectedness. As a corollary of our first main result, Corollary \ref{C1}, we show that all the universal coverings are Spanier covering. A Spanier covering is a covering $p:\wt{X}\lo X$ with $\pst=\psp$ which is universal as we have shown in \cite{P5}.

Among the recent works on studying the universal coverings which we are aware, we can point out the followings:\\
\textbf{G.R. Conner and J.W. Lamoreaux} \cite{La}: They studied the existence of simply connected universal covering spaces for separable metric spaces and subsets of the Euclidean plane.\\
\textbf{J.W. Cannon and G.R. Conner} \cite{CC}: They studied the relation of the simply connected universal covering of a separable, connected, locally path connected, one-dimensional metric space with algebraic properties of its fundamental group.\\
 \textbf{C. Sormani and G. Wei} \cite{SW1,SW2}: They studied the existence of universal cover for Gromov-Hausdorff limit of a sequence of manifolds.\\
\textbf{ H. Fischer, A. Zastrow} \cite{FZ}: They introduced a generalized universal covering which enjoys
most of the usual properties, with the possible exception of evenly covered neighborhood.\\
\textbf{J. Wilkins} \cite{Wil}: He studied universal coverings of compact geodesic spaces. Here, it should be mentioned that although there are some overlaps between our results and his, but his approach to these result is by using discrete homotopy theory and concerns himself with compact geodesic spaces while we use continuous path counterparts and the result are more general.

Another result of this paper, Theorem \ref{T2} (which we owe its proof to \cite[Theorem 4.4]{CC}) says that countability of $\frac{\pi_1(X,x)}{\psp}$ guaranties existence of the universal covering for a connected, locally path connected first countable space $X$. We can consider this theorem as a generalization of Mycielski's conjecture \cite{My} that is proved by Shelah \cite{Sh} and Pawlikowski \cite{Paw}. In fact, Shelah used very sophisticated model theory results and proved the Mycielski conjecture which state that: \textit{Fundamental group of a compact metric space which is connected and locally path connected is either finitely generated or
has the power of the continuum.} Pawlikowski has a follow-up result which replaces the model theory by sophisticated constructive set theory. Using the paragraph preceding Lemma 2 in \cite{Paw}, we can restate this Theorem (conjecture) as follow: \textit{If $X$ is compact metric space which is connected and locally path connected and $\pi_1(X,x)$ is countable, then $X$ has simply connected universal covering.} As a consequence of Theorem \ref{T2}, we show that by deletion compact metric hypothesis from Shelah Theorem, we will just lose simply connectedness of universal covering.

Our last paseo in universal covering spaces is about one point union of two space. The Griffiths space is an example of the one point union of two spaces with simply connected universal coverings which hasn't got simply connected universal covering (Example \ref{E1}).
At first, we present a Seifert-van Kampen type theorem for the fundamental group of the one point union of two spaces and then, using it and Theorem \ref{T1}, we prove that $X_1\vee X_2$ has a universal covering if and only if both of $X_1$ and $X_2$ have universal coverings.

%//////////////////////////////////////////////////////////////////////////////////////////////////////////////////////////////////////////////
%=============================================================================================================================================
%\\\\\\\\\\\\\\\\\\\\\\\\\\\\\\\\\\\\\\\\\\\\\\\\\\\\\\\\\\\\\\\\\\\\\\\\\\\\\\\\\\\\\\\\\\\\\\\\\\\\\\\\\\\\\\\\\\\\\\\\\\\\\\\\\\\\\\\\\\\\\\
\section{Definitions and terminology}
Throughout this article, all the homotopies between two paths are relative to end points,
$X$ is a connected and locally path connected space with the base point $x\in X$, and $p:\wt{X}\lo X$ is a covering of $X$ with $\ti{x}\in p^{-1}(\{x\})$ as the base point of $\wt{X}$. For a space $X$ and any $H\leq\pi_1(X,x)$, by $\wt{X}_H$ we mean a covering space of $X$ such that $\pst=H$, where $\ti{x}\in p^{-1}(x)$ and $p:\wt{X}_H\longrightarrow X$ is the corresponding covering map.

 E.H. Spanier \cite[\S 2.5]{Sp} classified path connected covering spaces of a space $X$ using some subgroups of the fundamental group of $X$, recently named Spanier groups (see \cite{FRV}). If $\U$ is an open cover of $X$, then the subgroup of $\pi_1(X, x)$ consisting of all homotopy classes of loops that can be represented by a product of the following type $$\prod\limits_{j=1}^{n}\al_j*\bt_j*\al_j^{-1},$$
where the $\al_j$'s are arbitrary paths starting at the base point $x$ and each $\bt_j$ is a loop inside one of the neighborhoods $U_i\in\mathcal{U}$,
is called the \emph{Spanier group with respect to $\U$}, and denoted by $\pi(\U,x)$ \cite{FRV,Sp}. For two open covers $\U,\V$ of $X$, we say that $\V$ refines $\U$ if for every $V\in\V$, there exists $U\in\U$ such that $V\sub U$.
%Using the properties of open covers and the definition of the Spanier groups with respect to open covers, we have the following facts which have been also remarked in \cite{Sp}.
%\begin{proposition}\label{p3}
%Let $\U,\V$ be open covers of a space $X$. Then the following statements hold.\\
%(i) If $\V$ refines $\U$, then $\pi(\V,x)\sub\pi(\U,x)$, for every $x\in X$.\\
%(ii) $\pi(\U,x)$ is a normal subgroup of $\pi_1(X,x)$.\\
%(iii) If $\al$ is a path in $X$, then $\vf_{[\al]}(\pi(\U,\al(0)))=\pi(\U,\al(1))$, where $\vf_{[\al]}([\bt])=[\al^{-1}*\bt*\al]$.
%\end{proposition}
\begin{definition}
We say that an open cover $\U$ of a space $X$ is $\pi$-stable if $\pi(\U,x)=\pi(\V,x)$, for every refinement $\V$ of $\U$ and $x\in X$.
\end{definition}

\begin{definition}\cite{FRV}
The Spanier group of a topological space $X$, denoted by $\psp$ is $\psp=\bigcap\limits_{open\ covers\ \U}\pi(\U,x),$ for an $x\in X$.
\end{definition}
Also, we can obtain the Spanier groups as follows: Let $\U,\V$ be open coverings of $X$, and let $\U$ be
a refinement of $\V$. Then since $\pi(\U,x)\sub\pi(\V,x)$, there exists an inverse limit of these Spanier groups, defined via the directed
system of all open covers with respect to refinement and it is $\psp$ (\cite{FRV}).

In the next definition, we follow \cite{P5}:
\begin{definition}
$(i)$ A space $X$ is called Spanier space if $\pi_1(X,x)=\psp$, for $x\in X$.\\
$(ii)$ A covering $p:\wt{X}\lo X$ is called Spanier covering if $\wt{X}$ is a Spanier space.
\end{definition}
A desirable fact in the category of coverings of a space $X$ is the existence of $\wt{X}_H$, for every subgroup $H\leq\pi_1(X,x)$. We characterize spaces with this property as follows.
\begin{definition}
 We call a topological space $X$ a coverable space if $\wt{X}_H$ exists, for every subgroup $H\leq\pi_1(X,x)$ with $\psp\leq H$.
\end{definition}
Note that the above notion does not depend on the point $x$. Also, since the image subgroups of all the coverings contain $\psp$ (\cite{P5}), eliminating the condition $\psp\leq H$ from the above definition is meaningless.
\begin{definition}
$(i)$ A point $x\in X$ is called regular if $X$ is semi-locally simply connected at $x$.\\
$(ii)$ A non-regular point $x$ is called wild if for every open neighborhood $U$ of $x$ there is a loop $\al$ in $U$ such that $[\al]\notin\psp$\\
$(iii)$ A non-regular point is called tame if it is not wild.
\end{definition}
For example, the common point of shrinking circles in the Hawaiian Earring is a wild point and the common point of shrinking circles in the Harmonic Archipelago is a tame point.
\begin{remark}
The readers should compare the above definition with Definition 4.5 of \cite{Wil}. A little change in terminology make two definitions equivalent. In fact, by the results of \cite{P5}, it is an easy exercise to show that a loop $\al$ in $X$ belongs to $\psp$ if and only if it belongs to the image subgroup of every covering of $X$.
\end{remark}
The majority of basic algebraic topology books who study covering theory, introduce semi-locally simply connected spaces which are famous because of having simply connected universal covering. Precisely, as Cannon and Conner mentioned in \cite[Lemma 7.8]{CC}, existence of the covering space $\wt{X}_H$ of $X$ for $H\leq\pi_1(X,x)$ is equivalent to that every point $y\in X$ has an open neighborhood $U$ such that $i_*\pi_1(U,y)\leq H$, where $i:U\hookrightarrow X$ is the inclusion. This coincidence is seen in \cite{P2,T2,P5}, where the authors introduced three type of new categorical universal coverings: small covering, small generated covering and Spanier covering. Therein, the equivalent condition for the existence of these coverings are named, respectively: semi-locally small loop space, semi-locally small generated space and semi-locally Spanier space.
\begin{definition}
 We call a space $X$ a semi-locally Spanier space if and only if for each $x\in X$ there exists an open neighborhood $U$ of $x$ such that $i_*\pi_1(U,x)\leq\pi_1^{sp}(X,x)$.
\end{definition}
The following theorems are main results of the paper.
\begin{theorem}\label{T1}
For a connected and locally path connected space $X$, the following statements are equivalent.\\
$(i)$ $X$ is coverable.\\
$(ii)$ $X$ has a universal covering space.\\
$(iii)$ $X$ has a $\pi$-stable open cover.\\
$(iv)$ $X$ is a semi-locally Spanier space.\\
$(v)$ $X$ has no wild point.\\
$(vi)$ $\psp$ is an open subgroup of $\pi_1^{\tau}(X,x)$.
\end{theorem}
After some topological criterions for the existence of universal coverings, an algebraic criteria is given as follows.
Also, this theorem is a generalization of Theorem 2.1 in \cite{La}.
\begin{theorem}\label{T2}
A connected, locally path connected and first countable space $X$ has a universal covering if $\frac{\pi_1(X,x)}{\psp}$ is countable. The converse is true when $X$ is separable metric.
\end{theorem}
Proposition \ref{p6}, the main technical result of the paper, is a Seifert-van Kampen type theorem for the fundamental group of the one point union $X_1\vee X_2$. Using this and Theorem \ref{T1}, we prove that
\begin{theorem}\label{T3}
Let $X$ be the one point union $X_1\vee X_2=\frac{{X_1}\cup {X_2}}{{x_1}\sim {x_2}}$, where $\{x_1\}$ and $\{x_2\}$ are closed in $X_1$ and $X_2$, respectively. Then $X$ has a universal covering if and only if $X_1$ and $X_2$ admit universal coverings.
\end{theorem}
%=========================================================================================================================================
%=========================================================================================================================================
\section{Propositions and proofs of the main results}
Although, the Spanier's brilliant book is the only book (as far as we know) that studies the existence of covering spaces from open cover viewpoint,
but some delicacies in this approach are evident and this caused that the influence of his book in new research becomes more. Therein, the main theorem is
\begin{theorem}(\cite[\S 2.5 Theorems 12,13]{Sp})\label{TS}
Let $X$ be a connected, locally path connected space and $H\leq\pi_1(X,x)$, for $x\in X$. Then there exists a covering $\pc$ such that $\pst=H$ if and only if there exists an open cover $\U$ of $X$ in which $\pi(\U,x)\leq H$.
\end{theorem}
For two open covers $\U,\V$ of $X$, we say that $\V$ refines $\U$ if for every $V\in\V$, there exists $U\in\U$ such that $V\sub U$.
Using the properties of open covers and the definition of the Spanier groups with respect to open covers, we have the following facts which have been also remarked in \cite{Sp}.
\begin{proposition}\label{ps}
Let $\U,\V$ be open covers of a space $X$. Then the following statements hold.\\
(i) If $\V$ refines $\U$, then $\pi(\V,x)\sub\pi(\U,x)$, for every $x\in X$.\\
(ii) $\pi(\U,x)$ is a normal subgroup of $\pi_1(X,x)$.\\
(iii) If $\al$ is a path in $X$, then $\vf_{[\al]}(\pi(\U,\al(0)))=\pi(\U,\al(1))$, where $\vf_{[\al]}([\bt])=[\al^{-1}*\bt*\al]$.
\end{proposition}
As the first observation, we have
\begin{proposition}\label{p1}
For a connected and locally path connected space $X$, let $H,K\leq\pi_1(X,x)$. Then $\wt{X}_H$ and $\wt{X}_K$ exist if and only if $\wt{X}_{H\cap K}$ exists.
\end{proposition}
\begin{proof}
By Theorem \ref{TS}, existence of $\wt{X}_H$ and $\wt{X}_K$ implies the existence of open covers $\U$ and $\V$ of $X$ such that $\pi(\U,x)\leq H$ and $\pi(\V,x)\leq K$. Let $\U\cap\V=\{U\cap V|U\in\U,\ V\in\V\}$ which is a refinement of $\U$ and $\V$. Hence $\pi(\U\cap\V)\sub\pi(\U)\sub H$ and $\pi(\U\cap\V)\sub\pi(\V)\sub K$ which implies that $\pi(\U\cap\V)\sub H\cap K$. Therefore, there exists $\wt{X}_{H\cap K}$. The converse is trivial.
\end{proof}
The above theorem shows that intersections of open covers of a space $X$ are important in the existence of new coverings of $X$. So it is interesting to find the impress of the intersection of all open covers. For this, we use the Spanier groups.
\begin{proposition}\label{p2} (\cite{P5}).
If $\pc$ is a covering of $X$, then $\psp\leq\pst$, for every $x\in X$.
\end{proposition}
It should be mentioned that the above proposition holds for $\pi_1^s(X,x)$ and $\psg$ (because of the inclusions $\pi_1^s(X,x)\leq\psg\leq\psp$) which their role in covering theory is studied in \cite{P2,T2}.

With a little change in terminology, the following result is well-known in the classical covering theory.
\begin{corollary}
Every connected, locally path connected and semi-locally simply connected space is coverable.
\end{corollary}
\begin{proposition}
Let $X$ be a connected, locally path connected and coverable space. Then $\psp$ is trivial if and only if $X$ is semi-locally simply connected.
\end{proposition}
\begin{proof}
Since $X$ is coverable, there exists a covering $\pc$ such that $\pst=\psp=1$ and hence $\wt{X}$ is simply connected which implies that $X$ is semi-locally simply connected. The converse holds by Proposition \ref{p2}.
\end{proof}
The Hawaiian Earring space, $HE$, is a famous example of a space which is not semi-locally simply connected. Also, the Spanier group of the Hawaiian Earring space is trivial since if $\U_n$'s are open covers of the Hawaiian Earring by open disk with diameter $1/n$, for every $n\in\N$, then $\pi_1^{sp}(HE,0)\leq\bigcap_{n\in\N}\pi(\U_n,0)=1$. Hence we have the following corollary.
\begin{corollary}
The Hawaiian Earring space is not coverable.
\end{corollary}
\begin{proposition}\label{p3}
A space $X$ is coverable if and only if $\wt{X}_{\psp}$ exists.
\end{proposition}
\begin{proof}
The necessity comes from the definition. For the sufficiency, let $\psp\leq H\leq \pi_1(X,x)$. By Theorem \ref{TS}, since $\wt{X}_{\psp}$ exists, there is an open cover $\U$ of $X$ such that $\pi(\U,x)\leq \psp$. Hence $\wt{X}_H$ exists.
\end{proof}
\begin{lemma}(\cite[\S 2.5 Lemma 11]{Sp}).\label{LS}
If $\pc$ is a covering such that $\pst=H$ and $\U$ is the open cover of $X$ by evenly covered open neighborhoods, then $\pi(\U)\leq H$.
\end{lemma}
\begin{proposition}\label{p4}
$(i)$ If $\U$ and $\V$ are two $\pi$-stable open covers of $X$, then $\pi(\U)=\pi(\V)$.\\
$(ii)$ The open cover $\U$ of $X$ is $\pi$-stable if and only if $\psp=\pi(\U,x)$.\\
$(iii)$ the covering space $\wt{X}_{\psp}$ exists if and only if there exists a $\pi$-stable open covering $\U$ of $X$.\\
\end{proposition}
\begin{proof}
The first step comes from definitions. For (ii), let $\U$ be a $\pi$-stable open cover of $X$. By definition $\psp\leq\pi(\U,x)$. For the reverse containment, let $\V$ be an arbitrary open cover of $X$. Then $\U\cap\V$ is a refinement of $\U$ and hence $\pi(\U)=\pi(\U\cap\V)\leq\pi(\V)$. Therefore $\pi(\U)\leq\psp$, as desired. The converse is trivial by definitions and part (i) of Proposition \ref{ps}. For (iii), assume $p:\wt{X}_{\psp}\lo X$ is a covering and $\U$ is the open cover of $X$ by evenly covered open neighborhoods. Since $p_*\pi_1(\wt{X}_{\psp},\ti{x})=\psp$, Lemma \ref{LS} implies that $\pi(\U,x)\sub\psp$ and hence the result holds by (ii). The converse holds by (ii) and Theorem \ref{TS}.
\end{proof}
The following theorem shows the importance of a universal covering for the existence of other coverings and vice versa.
\begin{theorem}\label{t5}
A space $X$ has a universal covering if and only if $X$ is coverable.
\end{theorem}
\begin{proof}
If $X$ is coverable, then by the definition, $\wt{X}_{\psp}$ exists. By Proposition \ref{p2}, $\wt{X}_{\psp}$ is a universal covering space and hence the result holds. Conversely, assume that $\pc$ is a universal covering of $X$ and $\pst=H$. We claim that for every open cover $\U$ of $X$, $H\leq\pi(\U,x)$. For, if $q:\wt{X}_{\pi(\U})\lo X$ is the covering such that $q_*\pi_1(\wt{X}_{\pi(\U}))=\pi(\U)$, then by the universal property of $\pc$
$$H=\pst\leq q_*\pi_1(\wt{X}_{\U})=\pi(\U).$$
 Hence $H\leq\psp$ which implies that $H=\psp$. Thus the covering space $\wt{X}_{\psp}$ exists and therefore by Proposition \ref{p3}, $X$ is coverable.
\end{proof}
\begin{proposition}\label{p5}
A space $X$ is semi-locally Spanier space if and only if there exists an open cover $\U$ of $X$ such that $\pi(\U,x)=\psp$, for every $x\in X$.
\end{proposition}
\begin{proof}
Use Proposition \ref{ps} (iii) and the definition of $\pi(\U,x)$.
\end{proof}
\ \ \\
\textbf{Proof of Theorem \ref{T1}.}\\
$(i)\Leftrightarrow (ii)$: Theorem \ref{t5}.\\
$(i)\Leftrightarrow (iii)$: Use Proposition \ref{p3} and Proposition \ref{p4}, (iii).\\
$(ii)\Leftrightarrow(iv)$: Use Proposition \ref{p5}, Proposition \ref{p4}, (iii) and Proposition \ref{p3}.\\
$(iv)\Leftrightarrow(v)$: By definition of wild point, the existence of wild point $x$ causes $X$ not to be semi-locally Spanier space at $x$ and vice versa.\\
$(i)\Leftrightarrow(vi)$: Use \cite[Theorem 2.1, Corollary 3.9]{T3}.$\hfill\Box$

By \cite{P5}, if $p:\wt{X}\lo X$ is a covering such that $\pst=\psp$, then $\wt{X}$ is a Spanier space and hence $p$ is a Spanier covering. Therefore we have
\begin{corollary}\label{C1}
All the universal covering spaces of connected and locally path connected spaces are Spanier space.
\end{corollary}
Recall that for a covering $p:\wt{X}\lo X$, we have $|p^{-1}(\{x\})|=[\pi_(X,x):\pst]$.\\
\ \ \\
\textbf{Proof of Theorem \ref{T2}.}\\
By Theorem \ref{T1}, it suffices to show that $X$ is semi-locally Spanier space. Let $y$ be fixed but arbitrary and $B_1\supseteq B_2\supseteq ...$ be a countable local basis at $y$. We will denote by $G_n$ the image of the natural map $\pi_1(B_n,y)\lo\pi_1(X,y)$. By \cite[Theorem 4.4]{CC}, the sequence $G_1\pi_1^{sp}(X,y)\supseteq G_2\pi_1^{sp}(X,y)\supseteq \cdots$ is eventually constant. We can choose $k\in\N$ large enough so that $G_k\pi_1^{sp}(X,y)=G_{k+1}\pi_1^{sp}(X,y)=\cdots$ and claim that $G_k\leq\pi_1^{sp}(X,y)$ which implies that $X$ is a semi-locally Spanier space. Let $[\al]\in G_k$ and $\U$ be an arbitrary open cover of $X$. If $U_y\in\U$ contains $y$, there exists  $m>k$ such that $B_m\sub U_y$ and then $\U_m:=\U\cup\{B_m\}$ is a refinement of $\U$. Obviously, $[\al]\in G_k\pi_1^{sp}(X,y)$ which implies $[\al]\in G_m\pi_1^{sp}(X,y)$ since $m>k$. Hence $[\al]=[\al_m][\gamma]$, where $[\al_m]\in G_m$ and $\gamma\in\pi_1^{sp}(X,y)$. Therefore $[\al]=[\al_m*\gamma]\in\pi(\U_m,y)\leq\pi(\U,y)$.

For the converse, let $p:\wt{X}\lo X$ be a universal covering of $X$ and assume by contradiction that $\frac{\pi_1(X,x)}{\psp}$ is uncountable. Then $p^{-1}(\{x\})$ is an uncountable subset of $\wt{X}$, which is a separable metric space by \cite[Theorem 4.1]{La}. Hence $p^{-1}(\{x\})$ contains a limit point of itself, contradicting the local homeomorphism property of covering maps.$\hfill\Box$

A restatement of the Mycielski's conjecture that is proved by Shelah \cite{Sh} and by Pawlikowski \cite{Paw} is that a connected, locally path connected, compact metric space with countable fundamental group has simply connected covering (which is universal covering). As a generalization, we can replace compact metric hypothesis by a weaker one first countablility and by passing from simply connected covering to universal covering.
\begin{theorem}(Generalized Shelah Theorem)
A connected, locally path connected and first countable space $X$ has a universal covering or $\pi_1(X,x)$ is uncountable.
\end{theorem}
By \cite[Theorem 5.1]{CC}, any free factor group of the fundamental group of a separable, locally path connected metric space has countable rank and hence is countable. Thus we have
\begin{corollary}
A connected, locally path connected separable metric space $X$ has a universal covering if $\frac{\pi_1(X,x)}{\psp}$ is free.
\end{corollary}
By the following corollary which is an explicit consequence of Theorems \ref{T1} and \ref{T2}, we are able to say easily that the fundamental group of a space with at least one wild point, like Hawaiian Earring, is uncountable.
\begin{corollary}

Let $X$ be connected, locally path connected and first countable. If $X$ has a wild point, then $\pi_1(X,x)$ is uncountable.
\end{corollary}
In the sequel, we concentrate on the fundamental group and the universal covering space of one point unions. At first, we introduce the following Seifert-van Kampen type formulation for the fundamental group of one point unions. It should be mentioned that according to Seifert-van Kampen theorem, the fundamental
group of the one point union of two spaces is naturally isomorphic to the free
product of their fundamental groups provided that each of them is first countable and locally simply
connected. But for the general spaces, this fails.
\begin{proposition}\label{p6}
Let $X$ be the one point union $X_1\vee X_2=\frac{{X_1}\cup {X_2}}{{x_1}\sim {x_2}}$ where $\{x_1\}$ and $\{x_2\}$ are closed in $X_1$ and $X_2$, respectively. If $U_i$ is a neighborhood of $x_i$ in $X_i$ for $i=1,2$, then
 $$\pi_1(X,*) =  < i_*\pi_1({X_1},{x_1}),j_*\pi_1({X_2},{x_2}),k_*\pi_1({U_1\vee U_2},*)>,$$
where $i: X_1\hookrightarrow X$, $j: X_2\hookrightarrow X$ and $k: U_1\vee U_2 \hookrightarrow X$ are inclusions and $*$ is the common point.
\end{proposition}

\begin{proof}
Let $\alpha : [0,1] \longrightarrow X$ be a loop at $* \in X$. First we define inductively $a_n \in \alpha ^{-1}(*)$ such that $0=a_0 \leq a_1 \leq a_2...\leq a_n \leq 1$ and $\alpha([a_{n-1},a_n])$ is a subset of $X_1$ or $X_2$ or $U_1\vee U_2$. Since $\{x_1\}$ and $\{x_2\}$ are closed in $X_1$ and $X_2$ respectively, $\{\alpha ^{-1}(X_1 \diagdown \{x_1\}),\alpha ^{-1}(X_2 \diagdown \{x_2\}),\alpha ^{-1}(U_1\vee U_2) \}$ is an open cover for the compact set $I$. Let $\lambda >0$ be the Lebesgue number for this cover. Choose $N \in \mathbb{N}$ such that $1/N \leq \lambda$. Put $a_0= 0$. Suppose $a_{n-1}$ has been chosen suitably. Now we obtain $a_n$ properly as follows:

If $a_{n-1}=1$, then put $a_n = 1$.
If $a_{n-1} \neq 1$ and $(\alpha^{-1}(*)) \cap (a_{n-1},min\{a_{n-1}+(1/N),1\}] \neq \varnothing$, then consider $a_n$ to be the maximum of the compact set $\{(\alpha^{-1}(*))\cap [a_{n-1},min\{a_{n-1}+(1/N),1\}] \}$. In this case since $a_n - a_{n-1} \leqslant (1/N)$, we have $\alpha([a_{n-1},a_n])$ is a subset of $X_1$ or $X_2$ or $U_1\vee U_2$.
If $a_{n-1} \neq 1$ and $(\alpha^{-1}(*))\cap (a_{n-1},min\{a_{n-1}+(1/N),1\}] = \varnothing$, then $a_{n-1}+(1/N) <1$ and put $a_n = min \{ (\alpha^{-1}(*))\cap [a_{n-1}+(1/N),1]\}$. In this case $\alpha((a_{n-1},a_n)) \subseteq X \diagdown \{*\} = (X_1 \diagdown \{x_1\}) \cup (X_2 \diagdown \{x_2\})$. Hence $(a_{n-1},a_n) \subseteq \alpha^{-1}(X_1 \diagdown \{x_1\}) \cup \alpha^{-1}(X_1 \diagdown \{x_1\})$. Since $(a_{n-1},a_n)$ is connected and $\alpha^{-1}(X_1 \diagdown \{x_1\}) \cap \alpha^{-1}(X_2 \diagdown \{x_2\}) = \varnothing$, $\alpha([a_{n-1},a_n])$ is a subset of $X_1$ or $X_2$.

Now we show that if $n\geq 2$ and $a_n\neq1$, then $a_n - a_{n-2} \geq 1/N$. If $(\alpha^{-1}(*))\cap (a_{n-2},min\{a_{n-2}+(1/N),1\}] = \varnothing$, then $a_{n-1} - a_{n-2} \geq 1/N$. If $(\alpha^{-1}(*))\cap (a_{n-2},min\{a_{n-2}+(1/N),1\}] \neq \varnothing$, then $(a_{n-1},a_{n-2}+(1/N)] \cap \alpha ^{-1}(*) = \varnothing$. Therefore $a_n \geq a_{n-2}+1/N$ and so there exists $k\in \mathbb{N}$ such that $a_k=1$. Hence $[\alpha]=[\alpha\circ\beta_1][\alpha\circ\beta_2]...[\alpha\circ\beta_k]$, where $\beta_i : I \rightarrow [a_{i-1},a_i]$ is an increasing linear homeomorphism for $i=1,2,\cdots ,k$. Note that for every $1\leq i\leq k$, $Im(\alpha\circ\beta_i)$ is a subset of $X_1$ or $X_2$ or $U_1\vee U_2$.
\end{proof}
\begin{lemma}\label{l2}
Let $X$ be a topological space with a base point $x$.\\
$(i)$ If $U$ is an element of an open cover $\U$ of $X$ which contains $x$, then $[\alpha] \in \pi(\U,x)$, for every loop $\alpha:I\lo U$ at $x$.\\
$(ii)$ If $Y$ is a subspace of $X$, then $[\alpha] \in \pi_1^{sp}(X,x)$, for every loop $\alpha$ such that $[\alpha] \in \pi_1^{sp}(Y,x)$.
\end{lemma}

\begin{proof}
$(i)$ Let $\alpha$ be a loop at $x$ in $U$. By the definition of $\pi(\U,x)$, $[\alpha] \in \pi(\U,x)$ since $U \in \U$.\\
$(ii)$ Let $[\alpha] \in \pi_1^{sp}(Y,x)$ and $\U = \{U_i | i\in I\}$ be an open cover for $X$. Then $\{U_i \cap Y | i\in I\}$ is an open cover for $Y$. By the definition there are paths $\alpha_j$ and loops $\beta_j$ at $\alpha_j(1)$ in $U_{i_j} \cap Y$ such that $\alpha$ is homotopic to $\alpha_1*\beta_1*\alpha_1^{-1}*\alpha_2*\beta_2*\alpha_2^{-1}...\alpha_n*\beta_n*\alpha_n^{-1}$ in $Y$ relative to $\{0,1\}$. Therefore $[\alpha] \in \pi(\U,x)$ which implies that $[\alpha] \in \pi_1^{sp}(X,x)$.
\end{proof}
\ \ \\
\textbf{Proof of Theorem \ref{T3}.}\\
Assume that $X_1$ and $X_2$ have universal coverings. Then they are semi-locally Spanier and there is a neighborhood $V_j$ of $x_j$ in $X_j$ such that $(i_j)_*\pi_1(V_j,x_j)\leq \pi_1^{sp}(X_j,x_j)$,
for $j=1,2$ and the inclusions $i_j:V_j\hookrightarrow X_j$. We show that
 $k_*\pi_1(V_1\vee V_2,\bar{x})\leq \pi_1^{sp}(X,\bar{x})$, where $k: V_1\vee V_2 \hookrightarrow X$ is
 the inclusion and $\bar{x}$ is the equivalence class of $x_1$ and $x_2$ in $X$. Let $\U$ be an open cover of $X$.
  There exists open set $U_1\vee U_2 \subseteq W \in \U$ such that $U_j$ is a neighborhood
   of $x_j$ in $X_j$, for $j=1,2$. By Proposition \ref{p6} $$\pi_1(V_1\vee V_2,*) =  < (l_1)_*\pi_1({V_1},{x_1}),(l_2)_*\pi_1({V_2},{x_2}),(l_3)_*\pi_1({(U_1\cap V_1)\vee (U_2\cap V_2)},\bar{x})>,$$
 where $l_j:V_j\lo V_1\vee V_2$, for $j=1,2$ and $l_3:(U_1\cap V_1)\vee (U_2\cap V_2)\lo V_1\vee V_2$ are inclusions. Since $(i_j)_*\pi_1(V_j,x_j)\leq \pi_1^{sp}(X_j,x_j)$,
 part (ii) of Lemma \ref{l2} implies that $(k_j)_*(i_j)_*\pi_1(V_j,x_j)\leq \pi_1^{sp}(X,\bar{x}) \leq \pi(\U,\bar{x})$,
  where $k_j:X_j\hookrightarrow X$, for $j=1,2$. Hence $k_*\pi_1(V_1\vee V_2,\bar{x})\leq \pi(\U,\bar{x})$ since $k_j\circ i_j=k\circ l_j$ for $j=1,2$ and $k_*(l_3)_*\pi_1({(U_1\cap V_1)\vee (U_2\cap V_2)},\bar{x})\sub \pi(\U,\bar{x})$ by  part (i) of Lemma \ref{l2}.
   For the converse, let  $x\in X_1$ and assume $q:X\lo X_1$ is a continuous map that is identity on $X_1$ and
    constant on $X_2$. Let $U$ be the open neighborhood of $x$ in $X$ where $i_*\pi_1(U,x)\leq\psp$,
     for $i:U\hookrightarrow X$. Then $V=q(U)$ is open in $X_1$. We claim that $j_*\pi_1(V,x)\leq\pi_1^{sp}(X_1,x)$,
      where $j:V\hookrightarrow X_1$. For if, let $\al:I\lo V$ be a loop and $\V$ be an open cover of $X_1$.
       Then $\U=q^{-1}(\V)$ is an open cover of $X$ and hence $[\al]\in\psp\leq\pi(\U,x)$.
       Therefore $q_*([\al])\in q_*(\pi(\U,x))=\pi(\V,x)$ which implies that $[\al]\in\pi_1^{sp}(X_1,x)$.$\hfill\Box$\ \\

In the following example we show that Theorem \ref{T3} does not hold for simply connected universal coverings.
\begin{example}\label{E1}
The cone on the Hawaiian Earring is a connected, locally path connected and semi-locally simply connected space and so has a simply connected covering space which is a universal covering space. But the double cone on the Hawaiian Earring does not have a simply connected universal covering space since it is not semi-locally simply connected but by Theorem 2.10 it has a categorical universal covering space which is a Spanier space.
\end{example}

\end{document}